\documentclass[12pt]{article}
\usepackage{ayv_paperstyle}

\usepackage{mathrsfs}  
\newcommand{\rdp}{\mathbb{R}^{d'}}
\newcommand{\prd}{\mathcal{P}^p(\mathbb{R}^{d})}

\title{Stability analysis of mean field type control system with major agent}
\author{Yurii Averboukh}
\date{}
\AtEndDocument{
	\vspace{10pt}
	\begin{tabular}{ll}
		Yurii Averboukh:  &
		Krasovskii Institute of Mathematics and Mechanics \\ &
		16 S. Kovalevskoi str., Yekaterinburg, Russia; \\ &
		Higher School of Economics,  \\ & 
		11 Pokrovsky Bulvar,  Moscow, Russia; \\ & 
		\email{averboukh@gmail.com}
	\end{tabular}
}

\begin{document}

\maketitle

\begin{abstract}
The paper is concerned with the study of a control system  consisting of one major agent and  many identical minor agents in the limit case when the number of agents tends to infinity. To study the limiting system we use the mean field approach.  We first prove the existence and uniqueness of the motion for such system consisting of finite dimensional  and mean field type subsystems. The main result of the paper is the stability of the motion  w.r.t. perturbations of dynamics, initial data and controls. To illustrate the general theory, we examine the Stackelberg game where the leader controls the major agent while the follower affects the minor agents. Using the general theory, we show the existence of a solution.
\msccode{49N80, 49J27, 91A65}
\end{abstract}

\section{Introduction}

The direct study of the controlled large system is realizable only in some rather specific  cases like one considered by Gabasov, Dmitruk and Kirillova in~\cite{Gabasov}. However, in many cases this problem can be simplified by passing to the limit system consisting of infinitely many agents. This assumption is quite natural in the analysis of economical problems~\cite{Gomes_MFG_applications}, crowd dynamics~\cite{Bellomo2012,Colombo2011,Colombo2005}, control of swarms of robots~\cite{Bullo2009}. The  study of limiting system consisting of infinitely many identical agents is a core concept of the theory of   mean field games~\cite{Huang5,Lions01,Lions02} that examines the system of identical players under assumption that each of them maximizes his/her own payoff,  mean field type control theory~\cite{ahmed_ding_controlld,Andersson_Djehiche_2011, Bensoussan_Frehse_Yam_book, Jimenez_Marigonda_Quincampoix} where the infinite agent system is governed by one decision maker. The close concept is the theory of control of continuity equation~\cite{Pogodaev}. In this case, the system of particles is affected by an external force. The same setting appears in the study of control problems within the probabilistic knowledge on initial conditions~\cite{Cardaliaguet_Quincampoux,Marigonda_Quincampoix}.

The paper is motivated by~\cite{Mean_field_spars_optimal} where the model of system consisting of infinite number of particles affected by the finite number of particle was examined in the case of interaction determined by some potential force. More generally, one can consider a system consisting of two parts. First one is finite dimensional and can be regarded as a major agent who interacts with the minor agents. The latter form the infinite-dimensional subsystem. We assume that the major agent as well as each minor agents can choose their controls. This form of the control system is quite general and can be used for analysis of various optimal control problems. For instance, one can assume that major and minor agents try to achieve a common goal. A different solution concept appears when we assume that the major agent has his/her own payoff while the minor agents play cooperatively to maximize a common payoff. 

In the paper, we primary restrict our attention to the study of quality properties of the mean field type control problem with the major agent considered on the finite time interval. We assume that the dynamics of each agent is given by a ordinary differential equation.
The main results of the paper is the stability of the motion of the examined system w.r.t.  the perturbations of dynamics, initial data and controls. Additionally, we derive the existence and uniqueness of the motion. The general result of the paper is  applied to the Stackelberg game where the leader control the finite-dimensional dynamics, while the follower can choose the individual controls of the minor agents who form the infinite-dimensional subsystem. Here we prove the existence of solution in this case.

The rest of the paper is organized as follows. In Section~\ref{sect:notation}, we present the general notation. Section~\ref{sect:motion_def} is concerned with the definition of the motion in the examined mean field type control system with major agent. The main results (existence and uniqueness theorem as well as stability property) are formulated in Section~\ref{sect:main}. We prove the existence of the motion in Section~\ref{sect:existence}. Section~\ref{sect:uniqueness_stability} is concerned with the analysis of uniqueness and stability of the motion. To this end, we estimate the distance between two motions. This auxiliary result gives both uniqueness and stability. Finally, in Section~\ref{sect:Stackelberg}, we examine the Stackelberg game where the leader controls the major agent and the minor agents are affected by the follower.

\section{General notation}\label{sect:notation}
In the paper, we use the following notation.
\begin{itemize}
	\item If $n$ is an integer number, $X_1,\ldots,X_n$ are sets, $i_1,\ldots,i_k$ are indices from $\{1,\ldots,n\}$, then we denote by $\operatorname{p}^{i_1,\ldots,i_k}$ an projection operator from $X_1\times\ldots X_n$ to $X_{i_1}\times \ldots \times X_{i_k}$, i.e.,
	\[\operatorname{p}^{i_1,\ldots,i_k}(x_1,\ldots,x_n)\triangleq (x_{i_1},\ldots,x_{i_k}).\]
	\item If $(\Omega,\Sigma)$ and $(\Omega',\Sigma')$ are measurable sets, $m$ is a probability on $\Sigma$, $h:\Omega\rightarrow\Omega'$ is a $\Sigma/\Sigma'$-measurable function, then $h\sharp m$ denotes  the push-forward measure defined by the rule: for $\Upsilon\in\Sigma'$,
	\[(h\sharp m)(\Upsilon)\triangleq m(h^{-1}(\Upsilon)).\]
	\item If $(X,\rho_X)$, $(Y,\rho_Y)$ are Polish sets, then $C(X;Y)$   stands for the set of all continuous  functions from $X$ to $Y$. The set of continuous and bounded functions from $X$ to $Y$ is denoted by $C_b(X;Y)$. It is endowed by usual $\sup$ norm. If $Y=\mathbb{R}$, then we will omit the second argument. 
	\item If  $(X,\rho_X)$ is a Polish space, then $\mathcal{M}(X)$ is a set of all Borel measures on $(X,\rho_X)$. We consider on $\mathcal{M}(X)$ the topology of narrow convergence, i.e.,  a sequence $\{m_n\}_{n=1}^\infty\subset \mathcal{M}(X)$ narrowly converges to $m$ if, for every $\phi\in C_b(X)$,
	\[\int_X\phi(x)m_n(dx)\rightarrow \int_X\phi(x)m(dx).\] The narrow convergence is metrizable. There are several metrics those generates this convergence. We will use one described in Appendix~B (see~\eqref{B:intro:metric_narrow}). 
	\item $\mathcal{P}(X)$ denote the set of all Borel probabilities on $X$, i.e.,
	\[\mathcal{P}(X)\triangleq \{m\in\mathcal{M}(X):m(X)=1\}.\]
	\item For two Polish spaces $(X,\rho_X)$ and $(Y,\rho_Y)$ and a measure $m\in\mathcal{M}(X)$, we denote by $\Lambda(X,m,Y)$ the set of measures on $X\times Y$ with marginal distribution on $X$ equal to $m$, i.e., $\alpha\in\mathcal{M}(X\times Y)$ lies in $\Lambda(X,m,Y)$ if, for every $\Upsilon\in \mathcal{B}(X)$, $\alpha(\Upsilon\times Y)=m(\Upsilon)$. Notice that, due to the disintegration theorem, there exists a system of probabilities $\alpha(\cdot|x)\in\mathcal{P}(Y)$ such that, for each $\phi\in C_b(X\times Y)$,
	\[\int_{X\times Y}\phi(x,y)\alpha(d(x,y))=\int_X\int_Y\phi(x,y)\alpha(dy|x)m(dx).\]
	\item If $p\geq 1$, then we denote by $\mathcal{P}^p(X)$ the set of all probabilities on $X$ with finite $p$-th moment, i.e., a probability $m$ lies in $\mathcal{P}^p(X)$ if, for some (equivalently, any) $x_*\in X$, the quantity
	\[\varsigma^p_p(m)\triangleq \int_X\rho_X^p(x,x_*)m(dx)\] is finite. In the following, we will use the designation $\varsigma_p(m)\triangleq [\varsigma_p^p(m)]^{1/p}$.
	\item The space $\mathcal{P}^p(X)$ is endowed with the so called $p$-th Wasserstein metric defined by the following rule: if $m_1,m_2\in\mathcal{P}^p(X)$, then 
	\[W_p(m_1,m_2)\triangleq \Big[\inf_{\pi\in\Pi(m_1,m_2)}\int_{X\times X}\rho_X^p(x_1,x_2)\pi(d(x_1,x_2))\Big]^{1/p}.\] Here $\Pi(m_1,m_2)$ is the set of all plans between $m_1,m_2$ those are  measures $\pi\in\Pi(X\times X)$ such that, for each Borel set $\Upsilon\subset X$, 
	\[\pi(\Upsilon\times X)=m_1(\Upsilon),\ \ \pi(X\times \Upsilon)=m_2(\Upsilon).\] Notice that the convergence in $W_p$ implies the narrow convergence.
	\item If $r>0$, then let  $\Gamma_r$ be  the set of all curves from $[0,r]$ to $\rd$, i.e.,
	\[\Gamma_r\triangleq C([0,r];\rd).\] Below we fix $T>0$, and omit the subindex in the case where $T=r$.
	\item We denote the evaluation operator from $\Gamma_r$ to $\rd$ by $e_t$. It acts by the rule
	\[e_t(x(\cdot))\triangleq x(t).\]
	\item If $c$ is a positive constant, $d$ is a natural number then we denote by $\mathbb{B}^d_c$ the ball $\{x\in \rd:\|x\|\leq c\}$. Further, let  $\mathbb{B}_{p,c}$ denote the set of probabilities $m$ on $\rd$ such that $\varsigma_p(m)\leq c$. 
\end{itemize}

\section{Mean field type control system with the major agent}\label{sect:motion_def}

The key object of the paper is a mean field type control system of minor agents who interacts with a major agents. It is assumed that the state of the minor agent is described by a $d$-dimensional vector, whereas the state of the major agent is given by a $d'$-dimensional vector. The dynamics of each minor agents  obeys the ODE
\begin{equation}\label{system:minor}
	\frac{d}{dt}x(t)=f(t,x(t),m(t),y(t),u(t),v(t)),
\end{equation} while the evolution of the major agent is given by
\begin{equation}\label{system:major}
	\frac{d}{dt}y(t)=g(t,y(t),m(t),v(t)).
\end{equation}
In equations~\eqref{system:minor},~\eqref{system:major}, 
\begin{itemize}
	\item $t\in [0,T]$ is a time, 
	\item $x(t)\in\rd$ stands for the state of a minor agent, 
	\item $m(t)$ is a current distribution of all minor agents, 
	\item $y(t)\in\mathbb{R}^{d'}$ describes the common state of major agent, 
	\item $u(t)\in U$ is a control used by the minor agent; $U$ is a control space for the minor agents;
	\item $v(t)\in V$ is the control of the major agent; $V$ is a control space for the major agent;
	\item $f:[0,T]\times\rd\times \prd\times \mathbb{R}^{d'}\times U\times V\rightarrow\rd$ determines the dynamics of the minor agents;
	\item $g :[0,T]\times \mathbb{R}^{d'}\times \prd\times V\rightarrow\mathbb{R}^{d'}$ is a dynamics function for the major player.
\end{itemize}

  We will assume that 
\begin{equation}\label{assump:f_sum}f(t,x,m,y,u,v)=f_I(t,x,m,y,u)+f_{II}(t,x,m,y,v).\end{equation}

In the paper, we use the concept of distribution of relaxed controls. It is defined in two steps. First, we define the relaxed control. Let $\lambda$ stand for the Lebesgue measure on $[0,T]$. A measure $\xi\in \Lambda([0,T],\lambda,U)$ is called a relaxed control of the minor agent. The set of all relaxed controls of the minor agents is denoted by $\mathcal{U}$. Analogously, $\mathcal{V}\triangleq \Lambda([0,T],\lambda,V)$ is the set of relaxed controls of the major agent.

To illustrate the meaning of the relaxed controls, let us consider the following control system
\[\frac{d}{dt}x(t)=b(t,x(t),u(t)).\] Here, as above $u(t)$ is from the set $U$. If $\xi\in\mathcal{U}$, $x_0$ is an initial state, then the corresponding motion is given by the solution of the following integral equation
\begin{equation}\label{eq:relaxed}x(t)=x_0+\int_{[0,t]\times U}b(\tau,x(\tau),u)\xi(d(\tau,u)).\end{equation}
Notice that, if one use a disintegration, then equation~\eqref{eq:relaxed} takes the form of initial value problem
\[\frac{d}{dt}x(t)=\int_Ub(t,x(t),u)\xi(du|t),\ \ x(0)=x_0.\]

We assume that the minor agents as a whole group use a distribution of relaxed controls. Assume that $m_0\in\mathcal{P}^p(\rd)$ is an initial distribution of minor agents. An element of the  set $\mathcal{A}[m_0]\triangleq \Lambda(\rd,m_0,\mathcal{U})$ is called a distribution of minor agents' controls compatible with the initial distribution $m_0$. 

To define the motion of the whole system we will use the following auxiliary operators. First, we assume that the motion of the major agent $y(\cdot)$ is given as well as the flow of probabilities $m(\cdot)$. If $x_0$ is an initial state of a minor agent, $\xi\in\mathcal{U}$, then we denote by $\mathscr{T}_0[f,m(\cdot),y(\cdot),x_0,\xi,\zeta]$ the function $x(\cdot)$ satisfying
\[\begin{split}
x(t)=x_0+\int_{[0,t]\times U}f_I(\tau,&x(\tau),m(\tau),y(\tau),u)\xi(d(\tau,u))\\&+\int_{[0,t]\times V}f_{II}(\tau,x(\tau),m(\tau),y(\tau),u)\zeta(d(\tau,v)).\end{split}\] Analogously, we define the motion of the major agent under assumption that the motion of the cloud of minor agents is given. If $m(\cdot)$ is a flow of probabilities, $y_0$ is an initial state of the major agent, while $\zeta$ is his/her relaxed control, then we denote by $\mathscr{T}^0[g,m(\cdot),y_0,\zeta]$ the solution of the equation
\begin{equation}\label{eq:y_relaxed}
	y(t)=y_0+\int_{[0,t]\times V}g(\tau,y(\tau),m(\tau),v)\zeta(d(\tau,v)).\end{equation}

\begin{definition}\label{def:motion} Let $m_0$ be an initial distribution of the minor agents, $\alpha\in\mathcal{A}[m_0]$ is the distribution of the minor agents' controls, $y_0$ be an initial state of the major agent, $\zeta$ be his/her relaxed control. A pair $(m(\cdot),y(\cdot))$ is a motion for the system~\eqref{system:minor},~\eqref{system:major} if there exists $\chi\in\mathcal{P}^p(\Gamma)$ such that
	\begin{itemize}
		\item $y(\cdot)=\mathscr{T}^0[g,m(\cdot),y_0,\zeta];$
		\item $m(t)=e_t\sharp\chi$;
		\item $\chi=\mathscr{T}_0[f,m(\cdot),y(\cdot),\cdot,\cdot,\zeta]\sharp\alpha$.
	\end{itemize}
	
\end{definition} We denote the motion generated by $f$, $g$, $m_0$, $y_0$, $\alpha$ and $\zeta$ by $\mathscr{X}[f,g,m_0,y_0,\alpha,\zeta]$.

\section{Main result}\label{sect:main} We will consider only the case when the sets $U$ and $V$ are metric compacts.

Let $\operatorname{SLLL}_0[A]$ be a set of continuous functions $f:[0,T]\times\rd\times\prd\times\rdp\times U\times V\rightarrow\rd$ satisfying \eqref{assump:f_sum} and such that 
\begin{itemize}
	\item $\|f_I(t,x,m,y,u)\|+ \|f_{II}(t,x,m,y,v)\|\leq A(1+\|x\|+\varsigma_p(m)+\|y\|)$;
	\item for any $c$ there exist constants $B_{I,c}$ and ${B_{II,c}}$ such that, for $t\in [0,T]$, $m_1,m_2\in\prd$, $\varsigma_p(m_1), \varsigma_p(m_2)\leq c$, $y_1,y_2\in \rdp$, $\|y_1\|, \|y_2\|\leq c$, $u\in U$,
		\[\begin{split}
		\|f_I(t,x_1,m_1,y_1,u)-f_I(t,&x_2,m_2,y_2,u)\|\\\leq B_{I,c}(\|x_1-x_2\|+&W_p(m_1,m_2)+\|y_1-y_2\|).\end{split}\]
			\[\begin{split}
		\|f_{II}(t,x_1,m_1,y_1,v)-f_{II}(t,&x_2,m_2,y_2,v)\|\\\leq B_{II,c}(\|x_1-x_2\|+&W_p(m_1,m_2)+\|y_1-y_2\|).\end{split}\]
	\end{itemize}
Below we will denote $B_c\triangleq B_{I,c}+B_{II,c}$. Obviously
	\[\begin{split}
		\|f_I(t,x_1,m_1,y_1,u)-f_I(t,&x_2,m_2,y_2,u)\|\\+ \|f_{II}(t,x_1,&m_1,y_1,v)-f_{II}(t,x_2,m_2,y_2,u,v)\|\\\leq B_c(\|x_1-x_2\|+&W_p(m_1,m_2)+\|y_1-y_2\|).\end{split}\]

Roughly speaking,  $\operatorname{SLLL}_0[A]$ is the set of dynamics for the minor agent those satisfy sublinear growth condition with the constant $A$ and are loaclly Lipschitz continuous with dynamics depending only on the ball in the space of states that is now product of the space of distribution of minor agents and state of the major agents.

We also introduce the  set of major agent's dynamics. In following, $\operatorname{SLLL}^0[A]$ is the set of continuous functions $g:[0,T]\times\rdp\times\prd \times V\rightarrow \rdp$ such that
\begin{itemize}
	\item $\|g(t,y,m,v)\|\leq A(1+\|y\|+\varsigma_p(m))$;
	\item for any $c_0$ there exists a constant $B_c'$ such that, for $t\in [0,T]$, $y_1,y_2\in\rdp$, $\|y_1\|,\|y_2\|\leq c$, $m_1,m_2\in\prd$, $\varsigma_p(m_1), \varsigma_p(m_2)\leq c$, $v\in V$,
	\[
	\|g(t,y_1,m_1,v)-f(t,y_2,m_2,v)\|\leq B_c'(\|y_1-y_2\|+W_p(m_1,m_2)).\]
\end{itemize} 

The following statement claims the existence theorem in the cased when $f\in\operatorname{SLLL}_0[A]$ while $g\in\operatorname{SLLL}^0[A]$.

\begin{theorem}\label{th:existence} Let $f\in\operatorname{SLLL}_0[A]$, $g\in\operatorname{SLLL}^0[A]$, where $A$ is a positive number. Then, for every $m_0\in\prd$, $y_0\in \rdp$, $\alpha\in\mathcal{A}[m_0]$, $\zeta\in\mathcal{V}$, there exists a unique motion.
\end{theorem}

To establish the stability result, let us introduce the notions of convergence of dynamics.

\begin{definition}\label{def:convergence}
	If $f,f'\in\operatorname{SLLL}_0[A]$, $c>0$, then we define the distance between restrictions of  $f$ and $f'$ on $[0,T]\times \rd\times\mathbb{B}_{p,c}\times\mathbb{B}^{d'}_c\times U$	by the rule:
	\[\begin{split}\operatorname{dist}_c(f,f')\triangleq \max\{\|f(t,x,&m,y,u,v)-f'(t,x,m,y,u,v)\|:\\&t\in [0,T],\, x\in \rd,\, m\in \mathbb{B}_{p,c},\, y\in \mathbb{B}^{d'}_c,\, u\in U, \, v\in V\}.\end{split}\]

	Analogously, if $g,g'\in\operatorname{SLLL}^0[A]$, $c>0$, then the distance between restrictions of  $g$ and $g'$ on $[0,T]\times\mathbb{B}_{p,c}\times\mathbb{B}^{d'}_c\times U$ is defined as follows:
	\[\begin{split}\operatorname{dist}_c(g,g')\triangleq \max\{\|g(t,&m,y,u)-g'(t,m,y,u)\|:\\&t\in [0,T],\, m\in \mathbb{B}_{p,c},\, y\in \mathbb{B}^{d'}_c,\, u\in U\}.\end{split}\]
\end{definition} 

Further, for $c_1,c_2,c_3$, set
\begin{equation}\label{intro:growth}
	\mathscr{G}(c_1,c_2,c_3)\triangleq (1+c_1+c_2)e^{2c_3T}.
\end{equation}

\begin{theorem}\label{th:stability} Let the sequences $\{f^k\}_{k=1}^\infty\subset\operatorname{SLLL}_0[A]$, $\{g^k\}_{k=1}^\infty\subset \operatorname{SLLL}^0[A]$, $\{m_0^k\}_{k=1}^\infty\subset\prd$, $\{y_0^k\}_{k=1}^\infty\subset \rdp$, $\{\alpha^k\}_{k=1}^\infty\subset\mathcal{P}([0,T]\times U)$, $\{\zeta^k\}_{k=1}^\infty\subset \mathcal{V}$, $\{m^k(\cdot)\}_{k=1}^\infty$, $\{y^k(\cdot)\}$ and elements $f\in\operatorname{SLLL}_0[A]$, $g\in \operatorname{SLLL}^0[A]$, $m_0\in\prd$, $y_0\in \rdp$, $\alpha\in\mathcal{P}([0,T]\times U)$, $\zeta\in \mathcal{V}$, $m(\cdot)$, $y(\cdot)$ are such that
	\begin{itemize}
		\item $\alpha^k\in\mathcal{A}[m_0^k]$;
		\item $\alpha\in\mathcal{A}[m_0]$;
		\item $(m^k(\cdot),y^k(\cdot))$ is a motion corresponding to the dynamics $f^k$, $g^k$, initial conditions $m_0^k$, $y_0^k$, distribution of relaxed controls $\alpha^k$ and the major agent control $\zeta^k$;
		\item the dynamics $f$, $g$, initial conditions $m_0$, $y_0$, distribution of the minor agents' control $\alpha$ and the major agent control $\zeta$ produces the motion $(m(\cdot),y(\cdot))$;
		\item the sequences  $\{m_0^k\}$, $\{y_0^k\}$, $\{\alpha^k\}$, $\{\zeta^k\}$ converge to the elements $m_0$, $y_0$, $\alpha$, $\zeta$;
		\item $\operatorname{dist}_c(f^k,f),\operatorname{dist}_c(g^k,g)\rightarrow 0$ as $k\rightarrow\infty$ 	
		for some $c$ that is greater than $\mathscr{G}(\|y_0\|,\varsigma_p(m_0),A)$.
	\end{itemize} Then, $\{(m^k(\cdot),y^k(\cdot))\}$ converge  to the pair $(m(\cdot),y(\cdot))$.
\end{theorem}

\section{Existence of the motion}\label{sect:existence}
To prove the existence of the motion $(m(\cdot),y(\cdot))$, we will use the method coming back to the Peano existence theorem. In this section, we assume that the dynamics, control of the major agent $\zeta$, distribution of minor agents' controls $\alpha$ and initial data are fixed.

We fix $N>$ and consider the pair $(\hat{\chi}^N,\hat{y}^N(\cdot))$, where $\hat{\chi}^N\in\mathcal{P}^p(\Gamma)$ that is constructed as follows. Put  $t_k^N\triangleq kT/N$, $k=0,\ldots,N$. On each time interval $[t_k^N,t_{k+1}^N]$, we define the motion $\hat{x}^N(\cdot,x_0,\xi)$ and the motion $\hat{y}^N(\cdot)$ by the following rules:
\begin{equation}\label{intro:hat_x_N}
	\begin{split}
	\frac{d}{dt}\hat{x}^N&(t)=\hat{x}^N(t_k^N)\\&+\int_{[t_k^N,t]\times U}f_I(\tau,\hat{x}(\tau-T/N),\hat{m}^h(\tau-T/N),\hat{y}^h(\tau-T/N),u)\xi(d(\tau,u))\\&+ \int_{[t_k^N,t]\times V}f_{II}(\tau,\hat{x}(\tau-T/N),\hat{m}^h(\tau-T/N),\hat{y}^h(\tau-T/N),v)\zeta(d(\tau,v)),\end{split}\end{equation}
\begin{equation}\label{intro:hat_y_N}\frac{d}{dt}\hat{y}^h(t)=\hat{y}^h(t_k^N)+\int_{[t_k^N,t]\times V}g(\tau,\hat{y}^h(\tau-T/N),\hat{m}^h(\tau-T/N),v)\zeta(d(\tau,v)).\end{equation} Here, we assume that, for $t<0$,
$\hat{x}^N(t,x_0,\xi)=x_0$, $\hat{y}^N(t)=y_0$, $\hat{m}^N(t)=m_0$. Additionally, the flow of probabilities $\hat{m}^N(\cdot)$ is defined stepwise. If the functions $\hat{x}^N(\cdot,\cdot,\cdot)$ are already defined on $[0,t_k^N]$, then we denote the corresponding  operator assigning to $x_0$ and $\xi$ the trajectory $\hat{x}^N(\cdot,x_0,\xi)\in \Gamma_{t_k^N}$ by $\widehat{\mathscr{T}}^N_k$. Set 
\begin{equation}\label{intro:hat_m_N}\hat{\gamma}^N_k\triangleq \widehat{\mathscr{T}}^N_k\sharp \alpha,\ \ \hat{m}^N(t)\triangleq e_t\sharp \hat{\gamma}^N_k. \end{equation}

\begin{lemma}\label{lm:ineq_last_two} The following estimate holds true
	\begin{equation}\label{ineq:last_two}
		\|\hat{y}^N(t)\|+\varsigma_p(\hat{m}^N(t))\leq (t+\|y_0\|+2\varsigma_p(m_0))e^{2At}.
	\end{equation} 
\end{lemma}
\begin{proof} 
	We prove  inequality~\eqref{ineq:last_two} inductively on each interval $[t_k^N,t_{k+1}^N]$. Since $f\in\operatorname{SLLL}_0[A]$, $g\in\operatorname{SLLL}^0[A]$,  we have that on $[t_0^N,t_1^N]$
	\begin{equation}\label{ineq:hat_x_N_0}
		\|\hat{x}^N(t,x_0,\xi)\|\leq \|x_0\|+A(1+\|x_0\|+\|y_0\|+\varsigma_p(m_0)))t,\end{equation}
	\begin{equation}\label{ineq:hat_y_N_0}\|\hat{y}^N(t)\|\leq \|y_0\|+ A(1+\|y_0\|+\varsigma_p(m_0))t.\end{equation} Estimate~\eqref{ineq:hat_x_N_0}, definition of the probability $\hat{m}^N(t)$ (see~\eqref{intro:hat_m_N}) and the Minkowski inequality  give that
	\begin{equation}\label{ineq:hat_m_N_0}\varsigma_p(\hat{m}^N(t))\leq 2\varsigma_p(m_0)+A(1+\|y_0\|+2\varsigma_p(m_0))t.\end{equation} 
	Summing~\eqref{ineq:hat_y_N_0} and~\eqref{ineq:hat_m_N_0}, we obtain~\eqref{ineq:last_two} on $[t_0^N,t_1^N]$.
	
	Now assume that~\eqref{ineq:last_two} holds true on $[0,t_k^N]$ let us prove them on $[t_{k}^N,t_{k+1}^N]$. Using the definition of $\hat{x}^N$ and $\hat{m}^N$, and the Minkowski's integral inequality, we obtain the following inequality:
	\begin{equation*}\label{ineq:m_N_k_plus}
		\varsigma_p(\hat{m}^N(t))\leq \varsigma_p(m^N(t_k^N))+A\int_{t_k^N}^t(1+2\varsigma_p(\hat{m}^N(\tau-T/N))+\|y^N(\tau-T/N)\|)d\tau.
	\end{equation*} Simultaneously,
	\begin{equation*}\label{ineq:y_N_k_plus}
		\|\hat{y}^N(t)\|\leq \|m^N(t_k^N)\|+A\int_{t_k^N}^t(1+\varsigma_p(\hat{m}^N(\tau-T/N))+\|y^N(\tau-T/N)\|)d\tau.
	\end{equation*} 
	Summing these inequalities and using the assumption that~\eqref{ineq:last_two} holds true on $[0,t_k^N]$, we arrive at the estimate
	\[\|\hat{y}^N(t)\|+\varsigma_p(\hat{m}^N(t))\leq (t_k^N+\|y_0\|+2\varsigma_p(m_0))e^{2At_k^N}+2A(t+\|y_0\|+2\varsigma_p(m_0))e^{2At_k^N}(t-t_k^N).\] This gives~\eqref{ineq:last_two} on $[t_k^N,t_{k+1}^N]$.

\end{proof}

\begin{lemma}\label{lm:ineq_x_N}
	For each $x_0\in \rdp$, $\xi\in\mathcal{U}$, one has
	\[\|\hat{x}^N(t,x_0,\xi)\|\leq (\|x_0\|+t)e^{At}+(t+\|y_0\|+2\varsigma_p(m_0))e^{2At}.\]
\end{lemma}
\begin{proof} As above, we will prove the desired estimate inductively.
	On $[t_0^N,t_1^N]$ the statement of lemma follows from~\eqref{ineq:hat_x_N_0}. Now assume that is fulfilled on $[0,t_{k}^N]$. For $t\in [t_k^N,t_{k+1}^N]$, we have that 
	\[\begin{split}
		\|\hat{x}^N(t,x_0,&\xi)\|\leq \|\hat{x}^N(t_k^N,x_0,\xi)\|\\&+\int_{t_k^N}^tA(1+\|\hat{x}^N(\tau-T/N,x_0,\xi)\|+ \hat{y}^N(\tau-T/N)\|+\varsigma_p(\hat{m}^N(\tau-T/N)))d\tau.\end{split}\] Using the assumption and Lemma~\ref{lm:ineq_last_two}, we conclude that 
	\[\begin{split}
		\|\hat{x}^N(t,x_0,\xi)\|\leq (\|x_0\|+t_k^N)&e^{At_k^N}+(t_k^N+\|y_0\|+2\varsigma_p(m_0))e^{2At_k^N}\\ +A[1+(\|x_0\|+t_k^N)&e^{At_k^N}+(t_k^N+\|y_0\|+2\varsigma_p(m_0)e^{2At_k^N}](t-t_k^N)\\ &+A(t_k^N+\|y_0\|+2\varsigma_p(m_0))e^{2At_k^N}](t-t_k^N).\end{split}\] This estimate  implies the statement of lemma on $[t_{k}^N,t_{k+1}^N]$.
\end{proof}

\begin{corollary}\label{corollary:hat_y_N_Lip} The functions $\hat{y}^N(\cdot)$ are Lipschitz continuous with the constant $L^1$ that does not depend on $N$.
\end{corollary}
\begin{proof} The statement of the corollary directly follows from the fact that  $g\in\operatorname{SLLL}^0[A]$ and Lemma~\ref{lm:ineq_last_two}.
\end{proof}

\begin{corollary}\label{corollary:hat_m_N_Lip} The functions $\hat{m}^N(\cdot)$ are Lipschitz continuous with the constant $L^2$ that does not depend on $N$, i.e., for $s,r\in [0,T]$,
	\[W_p(\hat{m}^N(s),\hat{m}^N(r))\leq L^2|r-s|.\] 
\end{corollary}
\begin{proof} Without loss of generality, one may assume that $s<r$.
	By the construction of the flow of probabilities $\hat{m}^N(\cdot)$ and the Minkowski's integral inequality, we have that
	\[\begin{split}W_p(\hat{m}^N(s),\hat{m}^N(r))\leq 
		\int_{\rd\times\mathcal{U}}\int_s^r A(&1+\hat{x}^N(t-T/N,x_0,\xi)\\+\|\hat{y}^N(t-&T/N)\|+\varsigma_p(\hat{m}(t-T/N)))dt\alpha(d(x_0,\xi)).\end{split}\] This and Lemmas~\ref{lm:ineq_last_two},~\ref{lm:ineq_x_N} implies the statement of the corollary.
\end{proof}

\begin{lemma}\label{corallary:hat_x_N_Lip} Let $C>0$. Then, each function $\hat{x}^N(\cdot,x_0,\xi)$ for $\|x_0\|\leq C$ is Lipschitz continuous with the constant $L_C^3$ that does not depend on $N$, $x_0$ and $\xi\in \mathcal{U}$.
\end{lemma}
\begin{proof}
	This corollary directly follows from the fact that $f\in \operatorname{SLLL}_0[A]$ and Lemma~\ref{lm:ineq_x_N}.
\end{proof}

In the following, we denote 
\[\hat{\chi}^N\triangleq \gamma^N_{N}.\]

\begin{lemma}\label{lm:chi_N_tight} The probabilities $\hat{\chi}^N$ are tight.
\end{lemma}
\begin{proof}
	Let $C>0$. We consider the set $K_C\subset\Gamma$ that consists of all curves $x(\cdot)$ such that \[\|x(\cdot)\|\leq (C+T)e^{AT}+(T+\|y_0\|+2\varsigma_p(m_0))e^{2AT}\] those a Lipschitz continuous with the constant equal to $L_C^3$. The set $K_C$ is compact by the Arzela–Ascoli theorem. Now, let 
	\[K'_C=\{x(\cdot,x_0,\xi):\|x_0\|\leq C,\xi\in\mathcal{U}\}.\] By Corollary~\ref{corallary:hat_x_N_Lip}, $K'_C\subset K_C$. This gives that 
	\begin{equation}\label{ineq:chi_N_K}\hat{\chi}^N(K_C')\leq \hat{\chi}^N(K_C).\end{equation} Now let us estimate $\hat{\chi}^N(K_C')$. Due to construction of $\hat{\chi}^N\triangleq \hat{\gamma}^N_N$ (see~\ref{intro:hat_m_N}), we have that
	\[\hat{\chi}^N(K_C')=\alpha\{(x_0,\xi):\|x_0\|\leq C\}=m_0(\mathbb{B}_C^d).\] This and~\eqref{ineq:chi_N_K} yield that 
	\[\hat{\chi}^N(K_C)\geq m_0(\mathbb{B}_C^d).\]
	Since $m_0$ is tight, we have that the family of probabilities $\{\hat{\chi}^N\}_{N=1}^\infty$ is tight.
\end{proof}

\begin{proof}[Proof of Theorem~\ref{th:existence}.Existence]
	By Lemma~\ref{lm:ineq_last_two} and Corollary~\ref{corollary:hat_y_N_Lip}, the sequence $\{\hat{y}^N(\cdot)\}_{N=1}^\infty$ is precompact. Furthermore, we claim that the sequence of measures $\{\hat{\chi}^N\}_{N=1}^\infty$ is precompact. Indeed, by Lemma~\ref{lm:chi_N_tight}, it is tight. Moreover, Lemma~\ref{lm:ineq_x_N} implies that
	\[\int_{\|x(\cdot)\|\leq C}\|x(\cdot)\|^p\hat{\chi}^N(d(x(\cdot)))\leq \int_{\|x_0\|\leq C}\|\hat{x}^N(\cdot,x_0,\xi)\|^p\alpha(d(x_0,\xi)).\] The right-hand side in this inequality  tends to zero uniformly w.r.t. $N$. Thus, the probabilities $\hat{\chi}^N$ have uniformly integrable $p$-th moments.
	
	Therefore (see \cite[Proposition 7.1.5]{Ambrosio}), there exist a sequence $\{N_k\}$ and a pair $(\chi,y(\cdot))$ such that 
	\[W_p(\hat{\chi}^{N_k},\chi)\rightarrow 0,\ \ \|\hat{y}^{N_k}(\cdot)-y(\cdot)\|\rightarrow 0,\text{ as }k\rightarrow \infty.\] Additionally, we put
	\[m(t)\triangleq e_t\sharp \chi.\] Now let us prove that $(m(\cdot),y(\cdot))$ is a motion for  system~\eqref{system:minor},~\eqref{system:major}. Notice that $e_t$ is a continuous operator from $\Gamma$ to $\rd$. Hence,
	\[m(t)=\lim_{k\rightarrow\infty}\hat{m}^{N_k}(t).\] This, construction of functions $\hat{y}^N(\cdot)$ (see~\eqref{intro:hat_y_N}) imply that $y(\cdot)$ satisfies~\eqref{eq:y_relaxed} and, thus,
	\[y(\cdot)=\mathscr{T}^0[g,m(\cdot),y_0,\zeta].\]
	
	It remains to prove that $\chi=\mathscr{T}_0[f,m(\cdot),y(\cdot),\cdot,\cdot,\zeta]\sharp\alpha$. To this end, we denote $\chi^\sharp\triangleq \mathscr{T}_0[f,m(\cdot),y(\cdot),\cdot,\cdot,\zeta]\sharp\alpha$ and prove that $W_p(\chi^\sharp,\hat{\chi}^{N_k})\rightarrow 0$. Passing to the limit in Lemma~\ref{lm:ineq_last_two}, we have that 
	\begin{equation}\label{ineq:y_m_limit}
		\|y(t)\|+\varsigma_p(m(t))\leq (t+\|y_0\|+\varsigma_p(m_0))e^{2At}.
	\end{equation} Since we fix dynamics, initial data and control, we simplify notation and put $x(\cdot,x_0,\xi)\triangleq \mathscr{T}_0[f,m(\cdot),y(\cdot),x_0,\xi,\zeta](\cdot)$. By  Gronwall's inequality,
	\begin{equation}\label{ineq:x_x_0_uniform}\|x(t,x_0,\xi)\|\leq  (\|x_0\|+At+A(t+\|y_0\|+\varsigma_p(m_0))e^{2At})e^{At}.\end{equation} Therefore, for each $C>0$, there exists a constant $L^4_C$ such that the functions $x(\cdot,x_0,\xi)$ are Lipschitz continuous with the constant $L^4_C$ whenever $\|x_0\|\leq C$. 
	
	Let $\varepsilon>0$.
	For each $C>0$, we have that 
	\begin{equation}\label{ineq:chi_sharp_chi_N}\begin{split}
			W_p^p(\chi^\sharp,\hat{\chi}^{N_k})\leq \int_{(\rd\setminus\mathbb{B}_C^d)\times \mathcal{U}}\|x(&\cdot,x_0,\xi)\|^p\alpha(d(x_0,\xi))\\&+ \int_{(\rd\setminus\mathbb{B}_C^d)\times \mathcal{U}}\|\hat{x}^{N_k}(\cdot,x_0,\xi)\|^p\alpha(d(x_0,\xi))\\&+
			\int_{\mathbb{B}_C^d\times\mathcal{U}}\|x(\cdot,x_0,\xi)-\hat{x}^{N_k}(\cdot,x_0,\xi)\|^p\alpha(d(x_0,\xi)).
		\end{split} 
	\end{equation} Choosing $C$ sufficiently large and using the tightness of $m_0$, we can assume that the first two integrals are less than $\varepsilon$. 
	
	Now, let us consider the third integral. We have that 
	\[\begin{split}\|x(t,x_0,\xi)-\hat{x}^{N_k}(t,x_0,\xi&)\|\\\leq \int_{[0,t]\times U}\|f_I(\tau,x(\tau&,x_0,\xi),y(\tau),m(\tau),u)\\- f_I(\tau,\hat{x}&{}^{N_k}(\tau-T/N_k,x_0,\xi),\hat{y}^{N_k}(\tau-T/N_k),\hat{m}^{N_k}(\tau-T/N_k),u)\| \xi(d(\tau,u))\\
		+\int_{[0,t]\times V}\|f_{II}(\tau&,x(\tau,x_0,\xi),y(\tau),m(\tau),v)\\- f_{II}&(\tau,\hat{x}^{N_k}(\tau-T/N_k,x_0,\xi),\hat{y}^{N_k}(\tau-T/N_k),\hat{m}^{N_k}(\tau-T/N_k),v)\| \zeta(d(\tau,v)).\end{split}\] Using the fact that the functions $\hat{y}^{N_k}(\cdot)$, $\hat{m}^{N_k}(\cdot)$ are Lipschitz continuous with the constants those do not depend on $N_k$, the functions $\hat{x}^{N_k}(\cdot,x_0,\xi)$ are also Lipchitz continuous whenever $\|x_0\|\leq C$ and the function $f$ is locally Lipschitz continuous, we have that 
	\[\begin{split}\|x(t,x_0,\xi)-\hat{x}^{N_k}(t,x_0,\xi&)\|\\\leq \int_{[0,t]\times U}\|f_I(\tau,x(\tau,&x_0,\xi),y(\tau),m(\tau),u)\\- f_I&(\tau,\hat{x}^{N_k}(\tau,x_0,\xi),\hat{y}^{N_k}(\tau),\hat{m}^{N_k}(\tau),u)\|\xi(d(\tau,u))\\+
	\int_{[0,t]\times V}\|f_{II}(\tau,&x(\tau,x_0,\xi),y(\tau),m(\tau),v)\\&- f_{II}(\tau,\hat{x}^{N_k}(\tau,x_0,\xi),\hat{y}^{N_k}(\tau),\hat{m}^{N_k}(\tau),v)\|\zeta(d(\tau,v))
		+a_1(C)/N_k.\end{split}\] Since the functions  $x(\tau,x_0,\xi)$ and $\hat{x}^{N_k}(\tau,x_0,\xi)$ are uniformly bounded whenever $\|x_0\|\leq C$ (see Lemma~\ref{lm:ineq_x_N} and~\eqref{ineq:x_x_0_uniform}), we use the fact that $f$ is locally Lipschitz continuous and obtain that 
	\[\|x(t,x_0,\xi)-\hat{x}^{N_k}(t,x_0,\xi)\|\leq a_2(C,N_k)/N_k.\] Here $a_2(C,N_k)$ is a function of $C$ that is bounded for each fixed $C>0$. Therefore,
	\[\int_{\mathbb{B}_C\times\mathcal{U}}\|x(\cdot,x_0,\xi)-\hat{x}^{N_k}(\cdot,x_0,\xi)\|^p\alpha(d(x_0,\xi))\leq \Big(\frac{a_2(C,N_k)}{N_k}\Big)^p.\] Choosing sufficiently large $k$, we can ensure that 
	the third integral in the right-hand side of~\eqref{ineq:chi_sharp_chi_N} is less than $\varepsilon$. Therefore, 
	\[W_p^p(\chi^\sharp,\hat{\chi}^{N_k})\leq 2\varepsilon,\] for sufficiently large $N_k$. This means that 
	$\hat{\chi}^{N_k}\rightarrow \chi^\sharp$. Thus, $\chi=\chi^\sharp=\mathscr{T}_0[f,m(\cdot),y(\cdot),\cdot,\cdot,\zeta]\sharp\alpha$.
\end{proof}

\section{Uniqueness and stability analysis}\label{sect:uniqueness_stability}
The purpose of this section is to prove the uniqueness part of Theorem~\ref{th:existence} and Theorem~\ref{th:stability}. This will rely on the key estimate proved below.

\subsection{Key estimate}
Let $A>0$, $f,f'\in\operatorname{SLLL}_0[A]$, $g,g'\in\operatorname{SLLL}^0[A]$, $m_0,m_0'\in\prd$, $y_0,y_0'\in \rdp$, $\alpha\in\mathcal{A}[m_0]$, $\alpha'\in\mathcal{A}[m_0']$, $\zeta,\zeta'\in \mathcal{V}$. Further, denote by $(m(\cdot),y(\cdot))$ the motion generated by the dynamics $f$, $g$, initial conditions $m_0$, $y_0$ and controls $\alpha$ and $\zeta$. Additionally, let  the dynamics $f'$, $g'$, initial conditions $m_0'$, $y_0'$ and controls $\alpha'$ and $\zeta'$ produce a motion $(m'(\cdot),y'(\cdot))$. Further, it is convenient  denote by $x(\cdot,x_0,\xi)$ a solution of the differential equation
\[\begin{split}
\frac{d}{dt}x(t)=x_0&+\int_{[0,t]\times U}f_I(\tau,x(\tau),m(\tau),y(\tau),u)\xi(d(\tau,u))\\&+\int_{[0,t]\times V}f_{II}(\tau,x(\tau),m(\tau),y(\tau),v)\zeta(d(\tau,v)).
\end{split}\] Here \[f(t,x,m,y,u,v)=f_I(t,x,m,y,u)+f_{II}(t,x,m,y,v).\]
Analogously, let $x'(\cdot,x_0,\xi)$ stand for a solution of the differential equation
\[\begin{split}
	\frac{d}{dt}x(t)=x_0&+\int_{[0,t]\times U}f_I'(\tau,x(\tau),m(\tau),y(\tau),u)\xi(d(\tau,u))\\&+\int_{[0,t]\times V}f_{II}'(\tau,x(\tau),m(\tau),y(\tau),v)\zeta'(d(\tau,v)).
\end{split}\] As above, we use the representation 
\[f'(t,x,m,y,u,v)=f_I'(t,x,m,y,u)+f_{II}'(t,x,m,y,v).\]

Let $\pi$ be a plan between $\alpha$ and  $\alpha'$.

Put 
\begin{equation}\label{key_estimate:intro:tilde_a}
	\tilde{a}(t)\triangleq \int_{[0,t]\times V}\psi(\tau,v)\zeta(d(\tau,v))-\int_{[0,t]\times V}\psi(\tau,v)\zeta'(d(\tau,v)),
\end{equation} while
\begin{equation}\label{key_estimate:intro:psi}
	\psi(t,v)\triangleq g(t,y(t),m(t),v).
\end{equation}  Furthermore, set 
\begin{equation}\label{key_estimate:intro:varphi}\varphi(t,x_0,\xi,u)\triangleq f_I(t,x(t,x_0,\xi),m(t),y(t),u).\end{equation} Additionally, we will use the notation
\begin{equation}\label{key_estimate:intro:hat_a}
	\hat{a}(t,x_0,\xi,x_0',\xi')\triangleq \Big\|\int_{[0,t]\times U} \varphi(\tau,x_0,\xi,u)\xi(d(\tau,u))-\int_{[0,t]\times U} \varphi(\tau,x_0,\xi,u)\xi'(d(\tau,u))\Big\|.
\end{equation} Finally, let 
\begin{eqnarray}\label{key_estimate:intro:chi}
	\eta(t,v)\triangleq \int_{\rd\times\mathcal{U}}f_{II}(t,x(t,x_0,\xi),m(t),y(t),v)\alpha(d(x_0,\xi)). 
\end{eqnarray}
\begin{equation}\label{key_estimate:intro:bar_a}
	\bar{a}(t)\triangleq \Bigg|\int_{[0,t]\times V}\eta(\tau,v)\zeta(d(\tau,v)) - \int_{[0,t]\times V}\eta(\tau,v)\zeta'(d(\tau,v)) \Bigg|.
\end{equation}

\begin{lemma}\label{lm:key_estimate}
	Let $c>0$ be such that $\mathscr{G}(\|y_0'\|,\varsigma_p(m_0'),A)\leq c$. Then, for each $t\in [0,T]$,
	\[\begin{split}
		W_p(m(t),m'&(t))+\|y(t)-y'(t)\|\\\leq C_0(W_p&(m_0,m_0')+\|y_0-y_0'\|)+C_1(\operatorname{dist}_c(f,f')+\operatorname{dist}_c(g,g'))\\+C_2&\int_{(\rd\times U)\times (\rd\times U)}\hat{a}(t,x_0,\xi,x_0',\xi')\pi(d(x_0,\xi,x_0',\xi'))+C_3\tilde{a}(t)+C_4\bar{a}(t).
	\end{split}
	\] Here $C_0$, $C_1$, $C_2$, $C_3$, $C_4$ are the constant depending only on $\|y_0\|$, $\|y_0'\|$, $\varsigma_p(m_0)$, $\varsigma_p(m_0')$.
\end{lemma}
\begin{proof}
	Due to the sublinear growth property (see~\eqref{ineq:y_m_limit}), we have that  the functions $\|y(\cdot)\|$, $\varsigma_p(m(\cdot))$,  $\|y'(\cdot)\|$, $\varsigma_p(m'(\cdot))$ are uniformly bounded. 
	Thus, there exists a constant $L^*$ such that 
	\[\begin{split}\|f_I(t,x(t,x_0,&\xi),m(t),y(t),u)-f_I(t,x'(t,x_0',\xi'),m'(t),y'(t),u)\|\\&\leq L^*(\|x(t,x_0,\xi)-x(t,x_0,\xi)\|+W_p(m(t),m'(t))+\|y(t)-y'(t)\|),\end{split}\]
	\[\begin{split}\|f_{II}(t,x(t,x_0,&\xi),m(t),y(t),v)-f_{II}(t,x'(t,x_0',\xi'),m'(t),y'(t),v)\|\\&\leq L^*(\|x(t,x_0,\xi)-x(t,x_0,\xi)\|+W_p(m(t),m'(t))+\|y(t)-y'(t)\|),\end{split}\]
	\[\|g(t,y(t),m(t),v)-f(t,y'(t),m'(t),v)\leq L^*(\|y(t)-y'(t))\|+W_p(m(t),m'(t))).\] 
	Moreover, notice that 
	\[\begin{split}
	\|f'_I(&t,x'(t,x_0',\xi'),m'(t),y'(t),u)+f'_{{II}}(t,x'(t,x_0',\xi'),m'(t),y'(t),v)\\&-f_I(t,x'(t,x_0',\xi'),m'(t),y'(t),u)-f_{II}(t,x'(t,x_0',\xi'),m'(t),y'(t),v)\|\leq \operatorname{dist}_c(f,f').\end{split}\]
	
	Therefore, we have that
	\begin{equation}\label{key_estimate:ineq:x_Lip}\begin{split}\|x(t,x_0,\xi)-x(t,&x_0',\xi')\|\leq \|x_0-x_0'\|+T\operatorname{dist}_c(f,f')\\&+ L^*\int_0^t\|x(\tau,x_0,\xi)-x(\tau,x_0,\xi)\|d\tau \\&+L^*\int_0^t(W_p(m(\tau),m'(\tau))+\|y(\tau)-y'(\tau)\|)d\tau+\bar{a}(t)\\&+
			\int_{[0,t]\times U} \varphi(\tau,x_0,\xi,u)\xi(d(\tau,u))-\int_{[0,t]\times U} \varphi(\tau,x_0,\xi,u)\xi'(d(\tau,u))
			.\end{split}\end{equation}  Here $\varphi$ is introduced by~\eqref{key_estimate:intro:varphi}, whereas $\bar{a}$ is determined by \eqref{key_estimate:intro:bar_a}.
	
	Applying  Gronwall's inequality to~\eqref{key_estimate:ineq:x_Lip}, we conclude that 
	\begin{equation*}\label{key_estimate:ineq:x_m}\begin{split}
			\|x(t,&x_0,\xi)-x(t,x_0',\xi')\|\leq e^{L^*t}\|x_0-x_0'\|+Te^{L^*t}\operatorname{dist}_c(f,f')\\&+ L^*e^{L^*t}\int_0^t(W_p(m(\tau),m'(\tau))+\|y(\tau)-y'(\tau)\|)d\tau+e^{L^*t}\bar{a}(t)+e^{L^*t}\hat{a}(t,x_0,\xi,x_0',\xi').\end{split}\end{equation*} Here, $\hat{a}$ is defined in~\eqref{key_estimate:intro:hat_a}.
	
	Integrating both sides w.r.t. to probability $\pi$, we have that 
	\begin{equation}\label{key_estimate:ineq:nu_Gronwall}\begin{split}
			W_p(m(t),m'(t))\leq e^{L^*t}W_p(&m_0,m_0')+e^{L^*t}T\operatorname{dist}_c(f,f')\\ L^*e^{L^*t}&\int_0^t(W_p(m(\tau),m'(\tau))+\|y(\tau)-y'(\tau)\|)d\tau+e^{L^*t}\bar{a}(t)\\&+e^{L^*T}\int_{(\rd\times U)\times (\rd\times U)}\hat{a}(t,x_0,\xi,x_0',\xi')\pi(d(x_0,\xi,x_0',\xi')).\end{split}\end{equation}
	
	Now let us use the estimate between $g$ and $g'$ as well as the Lipschitz continuity of the function $g$. We have that 
	\[\begin{split}
		\|y(t)-y'(t)\|\leq \|y_0&-y_0'\|+t\operatorname{dist}_c(f,f')\\&+ L^*\int_{0}^t(W_p(m(\tau),m'(\tau))+\|y(\tau)-y'(\tau)\|)d\tau+\tilde{a}(t),
	\end{split}\] where we use notation introduced in~\eqref{key_estimate:intro:tilde_a}. 
	
	Summing this inequality with~\eqref{key_estimate:ineq:nu_Gronwall}, we obtain the estimate
	\begin{equation}\label{key_estimate_ineq:nu_y_Gronwall}
		\begin{split}
			(W_p(m(t),m'(t))&+\|y(t)-y'(t)\|)\\\leq e^{L^*t}(W_p(&m_0,m_0'))+\|y_0-y_0'\|+	e^{L^*t}T(\operatorname{dist}_c(f,f')+\operatorname{dist}_c(g,g'))+
			\\+ (L^*e^{L^*t}&+L^*)\int_0^t(W_p(m(\tau),m'(\tau))+\|y(\tau)-y'(\tau)\|)d\tau\\+e^{L^*T}&\int_{(\rd\times U)\times (\rd\times U)}\hat{a}(t,x_0,\xi,x_0',\xi')\pi(d(x_0,\xi,x_0',\xi'))+e^{L^*t}\bar{a}(t)+\tilde{a}(t).
		\end{split}
	\end{equation} Using  Gronwall's inequality, we arrive at the statement of the lemma.
\end{proof}

\subsection{Uniqueness of the motion}
\begin{proof}[Proof of Theorem~\ref{th:existence}. Uniqueness]
	The uniqueness of the motion directly follows from Lemma~\ref{lm:key_estimate} and the facts that for $\zeta=\zeta'$, $\tilde{a}(t)=\bar{a}(t)=0$ while, if $\alpha=\alpha'$, $\pi$ is an optimal plan between them,
	\[\int_{(\rd\times U)\times (\rd\times U)}\hat{a}(t,x_0,\xi,x_0',\xi')\pi(d(x_0,\xi,x_0',\xi'))\] is also equal to zero.
\end{proof}

\subsection{Action of distribution of controls of minor agents}
In this subsection, we prove the following.

\begin{proposition}\label{prop:integral_convergence} Let 
	\begin{itemize}	
		\item $\varphi$ be defined by~\eqref{key_estimate:intro:varphi};
		\item $m_0\in\prd$, $\{m_0^k\}_{k=1}^\infty\in\prd$, $\alpha\in\mathcal{P}(\rd\times\mathcal{U})$, $\{\alpha^k\}\subset\mathcal{P}(\rd\times\mathcal{U})$ be such that $\alpha\in\mathcal{A}[m_0]$  $\alpha^k\in\mathcal{A}[m_0^k]$, while  $W_p(m_0,m_0^k)$ and $W_p(\alpha,\alpha^k)$ tend to zero;
		\item $\pi^k$ be an optimal plan between $\alpha$ and $\alpha^k$.
	\end{itemize}	
	Then, \[\begin{split}
		\int_{(\rd\times U)\times (\rd\times U)}\Big\|\int_{[0,t]\times U} \varphi&(\tau,x_0,\xi,u)\xi(d(\tau,u))\\-\int_{[0,t]\times U} \varphi(&\tau,x_0,\xi,u)\xi'(d(\tau,u))\Big\|\pi^k(d(x_0,\xi,x_0',\xi'))\rightarrow 0\text{ as }k\rightarrow\infty.\end{split}\]
\end{proposition}
\begin{proof}
	Recall that 
	\[\varphi(\tau,x_0,\xi,u)=f(\tau,x(\tau,x_0,\xi),y(\tau),m_0(\tau),u),\] where $x(\cdot,x_0,\xi)=\mathscr{T}_0[f,m(\cdot),y(\cdot),x_0,\xi]$. Since $f$ satisfies sublinear growth condition, as well as $x(\cdot,x_0,\xi)$ (see~\eqref{ineq:x_x_0_uniform}), we have that there exists a constant $C_5$ depending only on $A$, $y_0$ and $\varsigma_p(m_0)$ such that 
	\[\|\varphi(\tau,x_0,\xi,u)\|\leq C_5(1+\|x_0\|).\]
	
	Let $\varepsilon>0$. Choose sufficiently large $C$. We have that 
	\begin{equation}\label{conv_int:ineq:first}\begin{split}\int_{(\rd\times U)\times (\rd\times U)}\Big\|\int_{[0,t]\times U} \varphi(\tau,x_0,\xi,&u)\xi(d(\tau,u))\\-\int_{[0,t]\times U} \varphi(\tau,x_0,\xi,&u)\xi'(d(\tau,u))\Big\|\pi^k(d(x_0,\xi,x_0',\xi'))\\\leq 
			\int_{(\mathbb{B}_C^d\times U)\times (\mathbb{B}_C\times U)}\Big\|\int_{[0,t]\times U} &\varphi(\tau,x_0,\xi,u)\xi(d(\tau,u))\\-\int_{[0,t]\times U} &\varphi(\tau,x_0,\xi,u)\xi'(d(\tau,u))\Big\|\pi^k(d(x_0,\xi,x_0',\xi'))\\&+\int_{\rd\setminus\mathbb{B}_C^d}C_5(1+\|x_0\|)m_0(dx_0)\\&+\int_{\rd\setminus\mathbb{B}_C}C_5(1+\|x_0\|)m_0^k(dx_0).\end{split}
	\end{equation} Due to the fact that $W_p(m_0,m_0^k)\rightarrow\infty$, one can choose $c_*$ such that, for every $C>c_*$, the last two terms are less than given $\varepsilon$.
	
	Notice that, if $\|x_0\|\leq C$, then $\|\varphi(\tau,x_0,\xi,u)\|\leq C_5(1+C)$.
	
	To establish the fact that the first term in the right-hand side of~\eqref{conv_int:ineq:first} is less that $2\varepsilon$ for sufficiently large $k$, we consider the approximation by Lipschitz continuous functions. 
	
	Let us represent the function $\varphi$ in the coordinate-wide form, i.e., $\varphi(\tau,x_0,\xi,u)=(\varphi_i(\tau,x_0,\xi,u))_{i=1}^d$. We have that 
	\begin{equation}\label{conv_int:ineq:I_norms}
		\begin{split}
			\Big\|\int_{[0,t]\times U} &\varphi(\tau,x_0,\xi,u)\xi(d(\tau,u))-\int_{[0,t]\times U} \varphi(\tau,x_0,\xi,u)\xi'(d(\tau,u))\Big\|\\ \leq 
			&\sum_{i=1}^d\Big|\int_{[0,t]\times U} \varphi_i(\tau,x_0,\xi,u)\xi(d(\tau,u))-\int_{[0,t]\times U} \varphi_i(\tau,x_0,\xi,u)\xi'(d(\tau,u))\Big|.
		\end{split}
	\end{equation} Now notice that we consider the functions $\varphi_i(\tau,x_0,\xi,u)$ on a compact set $[0,T]\times\mathbb{B}_C^d\times\mathcal{U}\times U$. There exists a modulus of continuity $\varpi(\cdot)$ such that, if $\tau,\tau'\in [0,T]$, $u,u'\in U$ are such that $|\tau-\tau'|+d_U(u,u')\leq \delta$, then, for every $x_0\in\mathbb{B}_C^d$, $\xi\in\mathcal{U}$,
	\[|\phi_i(\tau,x_0,\xi,u)-\phi(\tau',x_0,\xi,u')|\leq \varpi(\delta).\] Now, let $N$ be a natural number. For, $n=0,\ldots, N$, we set $t_n^N\triangleq Tn/N$. Define 
	\begin{equation}\label{conv_int:inttro:tilde_varphi}
		\tilde{\varphi}_{i,n}^N(\tau,x_0,\xi,u)\triangleq \left\{\begin{array}{ll}
			\varphi_i(\tau,x_0,\xi,u), & \tau\in [0,t_n^N], \\
			\varphi_i(t_n,x_0,\xi,u)(t_{n+1}^l-t)N/T, & \tau\in [t_n^N,t_{n+1}^N], \\
			0, & t\in [t_{n+1}^N,T].
		\end{array}\right.
	\end{equation} If $t\in [t_n^N,t_{n+1}^N]$, then 
	\[|\varphi_i(\tau,x_0,\xi,u)-\tilde{\varphi}_{i,n}^N(\tau,x_0,\xi,u)|\leq C_5(1+C).\] Thus,
	\begin{equation}\label{conv_int:ineq:I_intervals}
		\begin{split}
			\Big|\int_{[0,t]\times U} \varphi_i(\tau,x_0,\xi,u)\xi(d(\tau,u))-\int_{[0,t]\times U} \varphi_i(\tau,x_0,\xi,u)\xi'(&d(\tau,u))\Big|\\ \leq 
			\Big|\int_{[0,T]\times U} \tilde{\varphi}_{i,j}^N(\tau,x_0,\xi,u)\xi(d(\tau,u))-\int_{[0,t]\times U} \tilde{\varphi}_{i,j}^N(\tau,x_0&,\xi,u)\xi'(d(\tau,u))\Big|\\&+2C_5(1+C)T/N.
	\end{split}\end{equation} Notice that each function $\tilde{\varphi}_{i,n}^N$ has a modulus of continuity w.r.t. $\tau$ and $u$ equal to $\tilde{\varpi}^N(\delta)\triangleq \varpi(\delta)\wedge (2C_5(1+C)\delta/N).$ Further, let $l$ be a natural number, $\hat{\varphi}_{i,n}^{N,l}$ be a $l$-Lipschitz continuous function that approximates $[C_5(1+C)]^{-1}\tilde{\varphi}_{i,n}^N$:
	\[\hat{\varphi}_{i,n}^{N,l}(\tau,x_0,\xi,u)\triangleq \inf\{[C_5(1+C)]^{-1}\tilde{\varphi}_{i,n}^N(\tilde{\tau},x_0,\xi,\tilde{u})+l(|\tilde{\tau}-\tau|+d_U(\tilde{u},u)):\tilde{\tau}\in [0,T],\tilde{u}\in U\}.\] Here $d_U$ stands for the metric on $U$, The function 
	$\hat{\varphi}_{i,n}^{N,l}$ is bounded by $C_5(1+C)$ by Proposition \ref{A:prop:Lip}. This proposition also gives that 
	\begin{equation}\label{conv_int:ineq:I_Lip}
		\begin{split}
			\Big|\int_{[0,T]\times U} \tilde{\varphi}_{i,n}^N(\tau,x_0,\xi,u)\xi(d(\tau,u))-&\int_{[0,t]\times U} \tilde{\varphi}_{i,n}^N(\tau,x_0,\xi,u)\xi'(d(\tau,u))\Big|\\ \leq C_5(1+C)\Big|\int_{[0,T]\times U} \hat{\varphi}_{i,n}^{N,l}&(\tau,x_0,\xi,u)\xi(d(\tau,u))\\&-\int_{[0,t]\times U} \hat{\varphi}_{i,n}^{N,l}(\tau,x_0,\xi,u)\xi'(d(\tau,u))\Big| +\tilde{\varpi}^N(2/l).
		\end{split}
	\end{equation} Now let us use the metric $\mathbf{d}$ (see~\eqref{B:intro:metric_narrow}) that metricize the narrow convergence for the case when $W=[0,T]\times U$. One can choose a function $\nu_l^j$ such that, for each $x_0$ and $\xi$,
	\[\|\nu_l^j(\cdot,\cdot)-\hat{\varphi}_{i,n}^{N,l}(\cdot,x_0,\xi,u)\|\leq 2^{-l}.\] Hence,
	\begin{equation*}
			\Big|\int_{[0,T]\times U} \hat{\varphi}_{i,n}^{N,l}(\tau,x_0,\xi,u)\xi(d(\tau,u))-\int_{[0,t]\times U} \hat{\varphi}_{i,n}^{N,l}(\tau,x_0,\xi,u)\xi'(d(\tau,u))\Big|\leq 2^l\mathbf{d}(\xi,\xi').
	\end{equation*} Using this,~\eqref{conv_int:ineq:I_intervals},~\eqref{conv_int:ineq:I_Lip}, we obtain the following estimate
	\begin{equation*}
		\begin{split}
			\Big|\int_{[0,t]\times U} \varphi_i(\tau,x_0,\xi,u)\xi(d(\tau,u))-&\int_{[0,t]\times U} \varphi_i(\tau,x_0,\xi,u)\xi'(d(\tau,u))\Big|\\ \leq 2C_5(1+C)T/N+C_5(1+&C_0){\varpi}(2/l)+\frac{2(C_5(1+C))^2}{Nl}+2^l(C_5(1+C))\mathbf{d}(\xi,\xi').
		\end{split}
	\end{equation*}
	
	Plugging this to estimate the first term in the right-hand side of ~\eqref{conv_int:ineq:first} and using~\eqref{conv_int:ineq:I_norms}, we conclude that 
	\[
	\begin{split}
		\int_{(\rd\times U)\times (\rd\times U)}\Big\|\int_{[0,t]\times U} \varphi(\tau,x_0,&u)\xi(d(\tau,u))\\-\int_{[0,t]\times U} \varphi(\tau,x_0,\xi,&u)\xi'(d(\tau,u))\Big\|\pi^k(d(x_0,\xi,x_0',\xi'))\\\leq 
		\varepsilon+2C_5(1+C)T/N&+C_5(1+C){\varpi}(2/l)+\frac{2(C_5(1+C))^2}{Nl}\\+2^l&	\int_{(\rd\times U)\times (\rd\times U)}\mathbf{d}(\xi,\xi')\pi^k(d(x_0,\xi,x_0',\xi')).
	\end{split}
	\] Further, \[\int_{(\rd\times U)\times (\rd\times U)}\mathbf{d}(\xi,\xi')\pi^k(d(x_0,\xi,x_0',\xi')) \leq W_p(\alpha,\alpha^k).\] Therefore,
	\[
	\begin{split}
		\int_{(\rd\times U)\times (\rd\times U)}\Big\|\int_{[0,t]\times U} \varphi(\tau,x_0,&u)\xi(d(\tau,u))\\-\int_{[0,t]\times U} \varphi(\tau,x_0,\xi,&u)\xi'(d(\tau,u))\Big\|\pi^k(d(x_0,\xi),x_0',\xi')\\\leq 
		\varepsilon+2C_5(1+C)T/N&+C_5(1+C){\varpi}(2/l)\\&+\frac{2(C_5(1+C))^2}{Nl}+2^l	W_p(\alpha,\alpha^k).
	\end{split} 
	\] Since $C$ was already chosen by $\varepsilon$, one may find $N$ and $l$ such that 
	\[2C_5(1+C)T/N+C_5(1+C){\varpi}(2/l)+\frac{2(C_5(1+C))^2}{Nl}\leq\varepsilon.\] Finally, choose $k_*$ such that, for every $k>k_*$, 
	\[W_p(\alpha,\alpha^k)\leq 2^{-l}\varepsilon.\] Hence, we have that 
	\[\begin{split}
		\int_{(\rd\times U)\times (\rd\times U)}\Big\|\int_{[0,t]\times U} \varphi(\tau,x_0,&u)\xi(d(\tau,u))\\-\int_{[0,t]\times U} \varphi(\tau,x_0,\xi,&u)\xi'(d(\tau,u))\Big\|\pi^k(d(x_0,\xi),x_0',\xi')\leq 3\varepsilon\end{split}\] whenever $k>k_*$.
\end{proof}

\subsection{Action of controls of the major agent} In this section, we prove the following statement.
\begin{proposition}\label{prop:action_major} Let $\psi$ be defined by~\eqref{key_estimate:intro:psi} and let a sequence $\{\zeta^k\}_{k=1}^\infty$ narrowly converge to $\zeta$. Then,
	\begin{equation}\label{action_major:convergence_psi}\int_{[0,t]\times V} \psi(\tau,v)\zeta^k(d(\tau,v))\rightarrow\int_{[0,t]\times V}\psi(\tau,v)\zeta(d(\tau,v))\end{equation}
		\begin{equation}\label{action_major:convergence_eta}\int_{[0,t]\times V} \eta(\tau,v)\zeta^k(d(\tau,v))\rightarrow\int_{[0,t]\times V}\eta(\tau,v)\zeta(d(\tau,v))\end{equation}
	 uniformly w.r.t. the time variable.
\end{proposition}
\begin{proof}
	We will prove only \eqref{action_major:convergence_psi}. Convergence \eqref{action_major:convergence_eta} is proved in the same way. 
	
	 Convergence \eqref{action_major:convergence_psi} is equivalent to  the fact that, for each $\varepsilon>0$, one can find sufficiently large $k$ such that,  for every $t\in [0,T]$,
	\begin{equation}\label{conv_zeta:ineq:main}
		\Big\|\int_{[0,t]\times V}\psi(\tau,v)\zeta^k(d(\tau,v))-\int_{[0,t]\times V}\psi(\tau,v)\zeta(d(\tau,v))\Big\|\leq\varepsilon.
	\end{equation} To this end,  we choose a natural $N$ and consider the time instants
	\[t_n^N\triangleq Tn/N, \ \ n=0,\ldots,N.\]
	Set
	\[\hat{\psi}^N_n(\tau,v)\triangleq \left\{\begin{array}{ll}
		\psi(\tau,v),& t\in [0,t_n^N],\\
		\psi(t_{n+1}^N,v)(t_{n+1}^N-t)N/T, & t\in [t_{n}^N,t_{n+1}^N], \\ 
		0, & t\in [t_{n+1}^N,T].
	\end{array}\right.\] Since $\psi$ is continuous, it is bounded by some constant $C_6$ (it depends on $y_0$ and $\varsigma_p(m_0)$). Therefore, for each $\zeta'\in\mathcal{V}$ and $t\in [t_{n}^N,t_{n+1}^N]$
	\begin{equation}\label{conv_zeta:ineq:psi_N}\Big\|\int_{[0,T]\times V} \hat{\psi}_n^N(\tau,v)\zeta'(d(\tau,v)) - \int_{[0,t]\times V}\psi(\tau,v)\zeta'(d(\tau,v))\Big\|\leq C_6T/N.\end{equation} 
	
	One can choose $N$ such that $C_6T/N\leq \varepsilon/4$. Further, since $\{\zeta^k\}$ narrowly converges to $\zeta$, one can find $k^*$ such that, for each $n=0,\ldots,N-1$, 
	\[\Big\|\int_{[0,T]\times V} \hat{\psi}_n^N(\tau,v)\zeta^k(d(\tau,v))-\int_{[0,T]\times V} \hat{\psi}_n^N(\tau,v)\zeta(d(\tau,v))\Big\|\leq \varepsilon/2.\]  This,~\eqref{conv_zeta:ineq:psi_N} and the choice of $N$ yield~\eqref{conv_zeta:ineq:main}. 
\end{proof}

\subsection{Proof of stability theorem}
First, we prove that the narrow convergence of the distributions of the minor agents relaxed controls implies the convergence within the Wassersein metric.
\begin{lemma}\label{lm:ditribution_convergence}
	Assume that 
	\begin{itemize}
		\item $\{m_0^k\}_{k=1}^\infty\subset\prd$, $m_0\in\prd$, while $W_p(m_0^k,m_0)\rightarrow0$ as $k\rightarrow \infty$;
		\item $\alpha^k\in\mathcal{A}[m_0^k]$, $\alpha\in \mathcal{A}[m_0]$;
		\item the sequence $\{\alpha^k\}_{k=1}^\infty$ narrowly converges to $\alpha$.
	\end{itemize} Then, $W_p(\alpha^k,\alpha)\rightarrow0$ as $k\rightarrow \infty$.
\end{lemma}
\begin{proof}
	Recall (see \cite[Proposition 7.1.5]{Ambrosio}) that the convergence in $W_p$ is equivalent to the facts that $\{\alpha^k\}_{k=1}^\infty$ narrowly converges to $\alpha$ and the probabilities $\alpha^k$ have uniformly integrable $p$-th moments. We will show the latter property.  Let $C>0$. The uniform integrability of $p$-th moment of $\alpha^k$ means that 
	\[\int_{\|x_0\|+\mathbf{d}(\xi,\xi^*)\geq C}(\|x_0\|+\mathbf{d}(\xi,\xi^*))^p\alpha^k(d(x_0,\xi))\rightarrow 0\] as $C\rightarrow\infty$ uniformly w.r.t. $k$. Hereinafter, $\xi^*$ is an element of $\mathcal{U}$.
	
	Due to compactness of $\mathcal{U}$ the function $\mathcal{U}\ni\xi\mapsto \mathbf{d}(\xi,\xi^*)$ is bounded by some constant~$C_7$. Therefore, since $\alpha^k\in\mathcal{A}[m_0^k]$, we have that
	\begin{equation}\label{stability:ineq:moments}\int_{\|x_0\|+\mathbf{d}(\xi,\xi^*)\geq C}(\|x_0\|+\mathbf{d}(\xi,\xi^*))^p\alpha^k(d(x_0,\xi))\leq \int_{\|x_0\|\geq C}(\|x_0\|^p+C_7)^pm_0^k(dx_0).\end{equation} Now recall that $W_p(m_0^k,m_0)\rightarrow 0$. This by \cite[Proposition 7.1.5]{Ambrosio} implies the uniform integrability of $p$-th moments of $m_0^k$. Therefore, we have that the right-hand side of~\eqref{stability:ineq:moments} converges to zero as $C\rightarrow\infty$ uniformly w.r.t. $k$. This implies the uniform integrability of $p$-th moments of $\alpha^k$ and the convergence of $\{\alpha^k\}_{k=1}^\infty$ to $\alpha$ in $W_p$.
\end{proof}

The stability result directly follows from the statement proved above.
\begin{proof}[Proof of Theorem~\ref{th:stability}]
	Due to Lemma~\ref{lm:ditribution_convergence}, we have that $W_p(\alpha^k,\alpha)\rightarrow 0$ as $k\rightarrow\infty$.
	
	Further, since $c>\mathscr{G}(\|y_0\|,\varsigma_p(m_0),A)$, one may, without loss of generality, assume that $c>\mathscr{G}(\|y_0^k\|,\varsigma_p(m_0^k),A)$
	Therefore, by Lemma~\ref{lm:key_estimate}, we have that 
	\begin{equation}\label{stability:ineq:main}\begin{split}
			W_p(m(t),m^k(t)&)+\|y(t)-y^k(t)\|\\\leq C_0(W_p&(m_0,m_0^k)+\|y_0-y_0^k\|) +C_1(\operatorname{dist}_c(f,f^k)+\operatorname{dist}_c(g,g^k))\\+C_2&\int_{(\rd\times U)\times (\rd\times U)}\hat{a}(t,x_0,\xi,x_0',\xi')\pi^k(d(x_0,\xi,x_0',\xi'))+C_3\tilde{a}^k(t)+C_4\bar{a}^k(t), 
		\end{split}
	\end{equation} where
	\begin{itemize}
		\item $C_0$, $C_1$, $C_2$, $C_3$, $C_4$ are constant determined by $A$, $y_0$ and $m_0$;
		\item \begin{equation*}\label{key_estimate:intro:tilde_a_k}
			\tilde{a}^k(t)\triangleq \int_{[0,T]\times V}\psi(\tau,v)\zeta(d(\tau,v))-\int_{[0,T]\times V}\psi(\tau,v)\zeta^k(d(\tau,v)),
		\end{equation*}  for
		$\psi(t,v)\triangleq g(t,y(t),m(t),v)$;
		\item 
		\begin{equation}\label{key_estimate:intro:bar_a_k}
			\bar{a}^k(t)\triangleq \Bigg|\int_{[0,t]\times V}\eta(\tau,v)\zeta(d(\tau,v)) - \int_{[0,t]\times V}\eta(\tau,v)\zeta^k(d(\tau,v)) \Bigg|\end{equation}
		with 
		\[\eta(t,v)\triangleq \int_{\rd\times\mathcal{U}}f_{II}(t,x(t,x_0,\xi),m(t),y(t),v)\alpha(d(x_0,\xi))\]
		\item  
		\[\begin{split}
		\hat{a}(t,&x_0,\xi,x_0',\xi')\\&\triangleq \Big\|\int_{[0,t]\times U} \varphi(\tau,x_0,\xi,u)\xi(d(\tau,u))-\int_{[0,t]\times U} \varphi(\tau,x_0,\xi,u)\xi'(d(\tau,u))\Big\|,\end{split}
		\] with \[\varphi(t,x_0,\xi,u)\triangleq f(t,x(t,x_0,\xi),m(t),y(t),u);\]
		\item $\pi^k$ is an optimal plan between $\alpha$ and $\alpha^k$.
	\end{itemize}
	The convergence to zero of the first two terms in the right-hand side of~\eqref{stability:ineq:main}  follows from the assumptions that $W_p(m_0^k,m_0)\rightarrow 0, \operatorname{dist}_c(f,f^k), $$\operatorname{dist}_c(g,g^k)\rightarrow 0$ and  $y_0^k\rightarrow y_0$. The convergence of the third term to zero is due to Proposition~\ref{prop:integral_convergence}. Finally, the forth and fifth terms tend to zero by Proposition~\ref{prop:action_major}.
\end{proof}

\section{Stackelberg game with mean field type dynamics and major agent}\label{sect:Stackelberg}
To illustrate the general theory we consider the model Stackelberg game, where the leader chooses the control of the major agent, i.e., he/she control the variable $v$, while the follower influences the minor agents. For simplicity, we assume only terminal payoffs, i.e.,
the leader wishes to maximize $\sigma_L(m(T),y(T))$, while the follower's payoff is equal to $\sigma_F(m(T),y(T))$. Additionally, we assume that the initial distribution of the minor agents $m_0$ and the initial state of the major agent $y_0$ are fixed.

We will use relaxation. This leads to the assumption that the set of the leader's controls is $\mathcal{V}$, while the follower's controls are from $\mathcal{A}[m_0]$. 

To introduce the notion of the Stackelberg solution, define, for $(m(\cdot),y(\cdot))=\mathscr{X}[f,g,m_0,y_0,\alpha,\zeta]$,
\[\Sigma_L(\alpha,\zeta)\triangleq \sigma_L(m(T),y(T)),\ \ \Sigma_F(\alpha,\zeta)\triangleq \sigma_F(m(T),y(T)).\] 

Further, for each $\zeta\in\mathcal{V}$, set
\[\mathscr{O}[\zeta]\triangleq \operatorname{Argmax}\{\Sigma_F(\alpha,\zeta):\alpha\in\mathcal{A}[m_0]\}.\]

\begin{definition}\label{def:Stackelberg}
	We say that $(\zeta^*, \alpha^*)$ forms the Stackelberg  solution in the game with dynamics~\eqref{system:minor},~\eqref{system:major}, the follower controlling the major agent and the leader who governs the behavior of the major agent if
	\begin{itemize}
		\item $\alpha^*\in\mathscr{O}[\zeta^*]$;
		\item for each $\zeta\in \mathcal{V}$,
		\[\Sigma_L(\alpha^*,\zeta^*)\geq \max_{\alpha\in\mathscr{O}[\zeta]}\Sigma_L(\alpha,\zeta).\]
	\end{itemize} 
\end{definition}

\begin{theorem}\label{th:Stackelberg} Assume that $\sigma_F,\sigma_L:\prd\times \rdp\rightarrow\mathbb{R}$ are continuous. Then, there exists a Stackelberg solution of the  game with dynamics~\eqref{system:minor},~\eqref{system:major}, the follower controlling the minor agents and the leader who plays for the major agent.
\end{theorem}
\begin{proof}
	First notice that $\mathcal{V}$ is compact. To show the compactness of $\mathcal{A}[m_0]$, we are to prove that it is tight. Indeed, choose $C>0$, and consider the compact  
	\[K_C\triangleq \{(x_0,\xi):\|x_0\|\leq C\}.\] For each $\alpha\in\mathcal{A}[m_0]$, we have that 
	\[\alpha(K_C)=m_0(K_C).\] The tightness of $m_0$ implies that one can  find $C$ such that $\alpha(K_C)\leq\varepsilon$ for each $\varepsilon$. This means the tightness of $\mathcal{A}[m_0]$. The seminal Prokhorov theorem gives the compactness of $\mathcal{A}[m_0]$.
	
	By Theorem~\ref{th:stability} the functions $\Sigma_F$ and $\Sigma_L$ are continuous. Thus, we reduce the   Stackelberg solution with mean field dynamics to the Stackelberg game with compact action spaces and continuous payoffs which always has a solution \cite{Stackelberg_properties}.
\end{proof}


	\section*{Appendix A. Approximation by Lipschitz functions}\label{appendix:approx}
\setcounter{theorem}{0}
\renewcommand{\thetheorem}{A.\arabic{theorem}}
\renewcommand{\thecorollary}{A.\arabic{corollary}}
\renewcommand{\theproposition}{A.\arabic{proposition}}
\setcounter{equation}{0}
\renewcommand{\theequation}{A.\arabic{equation}}
	
	In this Appendix, we recall the well-known technique of approximation of continuous functions by the Lipschitz one. 
	\begin{proposition}\label{A:prop:Lip}
		Let $(W,d_W)$ be a compact space, $\phi:W\rightarrow \mathbb{R}$ be a continuous function, $\omega(\cdot)$ be its modulus of continuity, $C'$ be an upper bound of $|\phi|$, and let $l$ be a natural number. Define
		\begin{equation}\label{A:intro:l_approx}\hat{\phi}_l(x)\triangleq \min\{\phi(z)+ld_X(w,z):z\in X\}.\end{equation} Then, $\hat{\phi}_l$ is $l$-Lipschitz continuous, bounded by $C'$ and
		\[|\phi(w)-\hat{\phi}_l(z)|\leq 2\omega\bigg(\frac{2C'}{l}\bigg).\]
	\end{proposition}
The proof of this statement directly follows from the definition and, thus, omitted.

	\section*{Appendix B. Distance on the space of measures}\label{appendix:metric}
	\setcounter{theorem}{0}
	\renewcommand{\thetheorem}{B.\arabic{theorem}}
	\renewcommand{\thecorollary}{B.\arabic{corollary}}
	\renewcommand{\theproposition}{B.\arabic{proposition}}
	\setcounter{equation}{0}
	\renewcommand{\theequation}{B.\arabic{equation}}
	In this section, we introduce a distance on the space of probabilities that metricize the narrow convergence. Let $(W,d_W)$ be a compact set, $\mathcal{M}(W)$ be a set of measures on $W$ endowed with the topology of narrow convergence. If $l$ is a natural number, then denote by $\{\nu_l^j\}_{j=1}^{J_l}$ a system of $l$-Lipschitz continuous function those form a $2^{-l}$-net for the compact of $l$-Lipschitz continuous functions bounded by $1$. If $\mu,\mu'\in\mathcal{P}(W)$, then set
	\begin{equation}\label{B:intro:metric_narrow}
		\mathbf{d}(\mu,\mu')\triangleq ||\mu|-|\mu'||+ \sum_{l=1}^\infty 2^{-l}\max_{j=1,\ldots,J_l}\Big|\int_W\nu^j_l(w)\mu(dw)-\int_W\nu^j_l(w)\mu'(dw)\Big|.
	\end{equation} Hereinafter, 
	\[|\mu|=\mu(W).\]
	
	\begin{proposition}\label{B:prop:metric} The function $\mathbf{d}$ is a distance of $\mathcal{M}(W)$ that metricize the narrow convergence on $\mathcal{M}(W)$.
	\end{proposition}
	\begin{proof}
		The symmetry axiom and the triangle inequality are obvious. Additionally, if $\mu=\mu'$, then $\mathbf{d}(\mu,\mu') =0$. Now we should prove the converse implication, i.e, we shall show that the equality $\mathbf{d}(\mu,\mu')$ yields the fact that $\mu=\mu'$. First, notice that, in this case, $|\mu|=|\mu'|$. Choose a continuous function $\phi:W\rightarrow\mathbb{R}$. Without loss of generality, we assume that is is bounded by $1$. Denote its modulus of continuity by $\omega(\cdot)$. Let $\hat{\phi}_l$ be a $l$-Lipschitz approximation of the function $\phi$ constructed by~\eqref{A:intro:l_approx}. By Proposition~\ref{A:prop:Lip}, we have that 
		\begin{equation}\label{B:ineq:int}\begin{split}
				\Big|\int_W\phi(w)&\mu(dw)-\int_W\phi(w)\mu'(dw)\Big|\\&\leq \Big|\int_W\hat{\phi}_l(w)\mu(dw)-\int_W\hat{\phi}_l(w)\mu'(dw)\Big|+2\omega(2/l)(|\mu|\wedge |\mu'|).\end{split}\end{equation}
		
		Due to the assumption that $\mathbf{d}(\mu,\mu')=0$ and the definition of $\mathbf{d}$, we have that
		\[\Big|\int_W\hat{\phi}_l(w)\mu(dw)-\int_W\hat{\phi}_l(w)\mu'(dw)\Big|\leq 2^{-l+1}(|\mu|\wedge |\mu'|).\]
		This and~\eqref{B:ineq:int} imply that \[\int_W\phi(w)\mu(dw)=\int_W\phi(w)\mu'(dw).\] Since $\phi$ was chosen arbitrarily, by the Riesz representation theorem, we have that $\mu=\mu'$. 
		
		Now let us show that $\mathbf{d}$ metricizes the narrow convergence. The fact that the narrow convergence implies convergence in $\mathbf{d}$ is obvious. To show the converse, consider a sequence $\{\mu^k\}$ such that $\mathbf{d}(\mu,\mu^k)\rightarrow 0$. By the definition of the metric $\mathbf{d}$, $|\mu^k|$ are bounded by some constant $c^1$. Let $l$ be a natural number and let $k^*$ be such that
		\[\mathbf{d}(\mu,\mu^k)\leq 2^{-2l}\] when $k\geq k^*$.
		Further, choose $\phi\in C(W)$. Without loss of generality, we assume that it is bounded by $1$. As above, let $\omega(\cdot)$ by the modulus of continuity of the function $\phi$. As above, let $\hat{\phi}_l$ be the $l$-Lipschitz  approximation of $\phi$ constructed by~\eqref{A:intro:l_approx}. We have that 
		\begin{equation}\label{B:ineq:int_k}\begin{split}
				\Big|\int_W\phi(w)&\mu(dw)-\int_W\phi(w)\mu^k(dw)\Big|\\&\leq \Big|\int_W\hat{\phi}_l(w)\mu(dw)-\int_W\hat{\phi}_l(w)\mu^k(dw)\Big|+2\omega(2/l)c^1. \end{split}\end{equation}
		Further, since $\{\nu_l^j\}_{j=1}^{J_l}$ is the $2^{-l}$ net in the space of $l$-Lipschitz continuous functions, there exists a number $j$ such that $\|\hat{\phi}_l-\nu_l^j\|\leq 2^{-l}$. Therefore,
		\begin{equation*}
			\begin{split}
				\Big|\int_W&\hat{\phi}_l(w)\mu(dw)-\int_W\hat{\phi}_l(w)\mu^k(dw)\Big|\\&\leq 
				\Big|\int_W\nu_l^j(w)\mu(dw)-\int_W\nu_l^j(w)\mu^k(dw)\Big|+c^12^{-l}\\&\leq 2^{l}\mathbf{d}(\mu,\mu^k)+c^12^{-l}.
			\end{split} 
		\end{equation*}
		The latter inequality directly follows from  the definition of the metric $\mathbf{d}$. Since $\mathbf{d}(\mu,\mu^k)\leq 2^{-2l}$, from~\eqref{B:ineq:int_k}, it follows that 
		\[\Big|\int_W\phi(w)\mu(dw)-\int_W\phi(w)\mu^k(dw)\Big|\leq 2\omega(2/l)c^1+ 2^{-l}+c^12^{-l}.\] This gives the convergence of $\int_W\phi(w)\mu^k(dw)$ to $\int_W\phi(w)\mu(dw)$.

	\end{proof}



\end{document}